%% file: MMZ.tex
\documentclass{amsart}
\usepackage{amsmath,amssymb,amsthm}
\usepackage{microtype}

\usepackage{orcidlink,enumitem}

\allowdisplaybreaks

\usepackage[capitalize,nameinlink]{cleveref}
\Crefname{enumi}{}{}
\Crefname{subsection}{Subsection}{Subsections}

\theoremstyle{plain}
\newtheorem{theorem}{Theorem}[section]

\newtheorem{proposition}[theorem]{Proposition}
\newtheorem{problem}[theorem]{Problem}
\newtheorem{fact}[theorem]{Fact}

\theoremstyle{definition}
\newtheorem{definition}[theorem]{Definition}
\newtheorem{example}[theorem]{Example}

\theoremstyle{remark}

\newcommand{\meq}{\mathop{=}\limits^{\mu}}

\makeatletter
\def\cref@thmoptarg[#1]#2#3#4{%
    \ifhmode\unskip\unskip\par\fi%
    \normalfont%
    \trivlist%
    \let\thmheadnl\relax%
    \let\thm@swap\@gobble%
    \thm@notefont{\fontseries\mddefault\upshape}%
    \thm@headpunct{.}
    \thm@headsep 5\p@ plus\p@ minus\p@\relax%
    \thm@space@setup%
    #2
    \@topsep \thm@preskip               
    \@topsepadd \thm@postskip           
    \def\@tempa{#3}\ifx\@empty\@tempa%
      \def\@tempa{\@oparg{\@begintheorem{#4}{}}[]}%
    \else%
      \refstepcounter[#1]{#3}
      \@namedef{cref@#3@alias}{#1}
      \def\@tempa{\@oparg{\@begintheorem{#4}{\csname the#3\endcsname}}[]}%
    \fi%
    \@tempa}%
\makeatother

\begin{document}
\title[Operators arising from invariant measures]{Operators arising from invariant measures under some class of multidimensional transformations}

\author{Oleksandr V. Maslyuchenko\orcidlink{0000-0002-1493-9399}, Janusz Morawiec\orcidlink{0000-0002-0310-867X}, 
and Thomas Zürcher\orcidlink{0000-0001-9179-9521}
}
\address{Institute of Mathematics, University of Silesia, Bankowa 14, 40-007 Katowice, Poland}\thanks{Corresponding author: Janusz Morawiec. Email: janusz.morawiec@us.edu.pl}



\begin{abstract}
We investigate a linear operator associated with a functional equation that arises from studying some class of invariant measures under multidimensional transformations. By examining its iterates, we derive an explicit solution formula for the functional equation in some class of functions and establish a result on the existence of an absolutely continuous invariant measure under a multidimensional transformation that can be viewed as a generalization of classical $p$\nobreakdash-adic maps to higher dimensions.
\end{abstract}

\maketitle


\section{Introduction}\label{section1}

The theory of dynamical systems is a concept that provides a unified mathematical framework for understanding how systems evolve over time. It has broad applicability across nearly every scientific field, including physics, biology, economics, and engineering (see, e.g., \cite{RH12, N13, JR15, A18, K23, A24, W24}).

An invariant measure is a probability measure that remains unchanged under the dynamics of a system. Invariant measures are crucial tools for describing the statistical properties of orbits, such as time averages, without the need to track individual trajectories.
They are fundamental in ergodic theory, where they help classify systems according to their mixing behaviour and randomness. In many systems, especially chaotic ones, invariant measures provide a way to analyse the \lq\lq typical\rq\rq\ behaviour of almost all points with respect to the measure. This leads to significant results such as the ergodic theorem and the existence of physical measures. These properties make invariant measures central objects of interest in the theory of dynamical systems (see, e.g., \cite{LM94,Z96,C02,D03,Y08,C20}).

In the study of invariant measures, appropriate operators are often used, particularly when investigating absolutely continuous invariant measures (see, e.g., \cite{LY73,B00,GB03,BG05,I12}). These measures are especially important because they describe how probability densities evolve.

Continuous and singular invariant measures are also interesting from the purely mathematical point of view. However, unlike the absolutely continuous case, no general method exists for their systematic investigation. An effective approach to studying invariant measures (including continuous and singular ones) is to reformulate the problem in terms of a functional equation in a single variable and analyse its solutions within the class of all probability distribution functions. While this method is not new (see, e.g., \cite{M85,N91}), to the best of our knowledge, it has only been employed sporadically in the literature (see \cite{MZ22,MS23}). 

The dyadic transformation is one of the maps for which a wide class of continuous and singular invariant measures has been obtained by solving the \emph{Matkowski--Wesołowski problem} (see \cite{MZ18,MZ21,MZ25}). This problem was posed in 1985 by Janusz Matkowski (see \cite{M85}) and independently, in an equivalent form during the 47th International Symposium on Functional Equations in 2009 by Jacek Wesołowski (see \cite{MZ18}). It consists in finding non-linear solutions of the \emph{Matkowski--Wesołowski functional equation} 
\begin{equation}\label{MWeq}
f(x)=f\left(\tfrac{x}{2}\right)-f(0)+f\left(\tfrac{x+1}{2}\right)-f\left(\tfrac{1}{2}\right)	
\end{equation}
in the class  $\mathcal D$ of all monotone and continuous functions $f\colon[0,1]\to\mathbb R$. Further results concerning extended versions of the Matkowski--Wesołowski functional equation in the one-dimensional setting can be found in \cite{MZ19}.

Since iteration is a fundamental technique for solving functional equations in a single variable and iterates usually appear in the formulas for their solutions (see \cite[Section~0.3]{KCG90}), it is natural to associate with the Matkowski--Wesołowski functional equation the \emph{Matkowski--Wesołowski operator} $\mathbb M\colon\mathcal D\to\mathcal D$ defined by 
\begin{equation}\label{MWop}
\mathbb M f(x)=f\left(\tfrac{x}{2}\right)-f(0)+f\left(\tfrac{x+1}{2}\right)-f\left(\tfrac{1}{2}\right).
\end{equation}
A counterpart of this operator has been used in \cite{MZa} to determine a general formula for probability distribution functions of all invariant measures under the $p$-addic transformations. 

The aim of this paper is to study a multidimensional version of the Matkowski--Wesołowski functional equation associated with a certain class of multidimensional transformations. The central focus is the analysis of the corresponding multidimensional Matkowski--Wesołowski operator arising from these transformations. Moreover, motivated by the study of invariant measures, we conclude with a result on the existence and uniqueness of absolutely continuous invariant measures for the class of transformations under consideration.


\section{Preliminaries}\label{section2}
We begin by fixing the main assumptions and notations used throughout this paper.

Assume that $r,s\in\mathbb N$ and set $K=\{1,\ldots,s\}$, $N=\{1,\ldots,r\}$, and $L=K\times N$.

We put $V=\mathbb R^K$. For any $x\in V^N$ we use the following double notation: $x=(x_n)_{n\in N }= (x_l)_{l\in L}$, where $x_n=(x_{n,k})_{k\in K}$ and $x_l=x_{n,k}$ if $l=(n,k)$, that is, we identify $V^N$ with $\mathbb{R}^L$. 

We fix a set $D\subseteq V^N$ such that $0=(0,\ldots,0)\in D$. Furthermore, we fix a countable set $I$ (finite or infinite) containing at least two elements and an indexed family $\gamma=(\gamma^i)_{i\in I}$ of self-mappings of the set $D$.

We denote by $\mu_n$ the Lebesgue measure on $\mathbb{R}^n$; when the dimension is clear from the context, we simply write $\mu$.
Recall that an indexed family $(A^i)_{i\in I}$ of Lebesgue measurable subsets of $\mathbb R^n$ is called \emph{$\mu$\nobreakdash-almost disjoint} if $\mu(A^i\cap A^j)=0$ whenever $i\neq j$. Given a $\mu$\nobreakdash-almost disjoint indexed family $(A^i)_{i\in I}$ of subsets of $\mathbb R^n$, we denote its union by $\bigsqcup_{i\in I}^\mu A^i$.	

Given $A,B\subseteq\mathbb R^n$, we write $A\meq B$ (and say that $A=B$ \emph{$\mu$\nobreakdash-almost everywhere}) if $\mu(A\triangle B)=0$, where $\triangle $ denotes the symmetric difference operator of sets.

We denote by $[x,y]$ the segment joining points $x,y\in V$. For all $a=(a_n)_{n\in N}\in V^N$ and $b=(b_n)_{n\in N}\in V^N$, we define $P(a,b)=\prod_{n\in N}[a_n,b_n]$.

The family of all Borel sets in a topological space $X$ is denoted by $\mathcal{B}(X)$. 

Furthermore, we denote by $\|\cdot\|$ the Euclidean norm, and by $|S|$ the cardinality of a set $S$.

Let $M\subseteq N$. 

We denote by $x|_M$ the restriction of $x\in V^N$ to $M$. Thus, if $x=(x_n)_{n\in N}\in V^N$, then $x|_M=(x_n)_{n\in M}\in V^M$, $x|_{N\setminus M}=(x_n)_{n\in N\setminus M}\in V^{N\setminus M}$, and $x|_M\cup x|_{N\setminus M}=x$.

For a function~$f$ defined on $D$ and all $x\in V^M$, we use the notation
\begin{equation*}
f_x(y)=f(x\cup y)\quad \text{for every }y\in V^{N\setminus M}\text{ such that }x\cup y\in D.
\end{equation*}
So, the function~$f_x$ is defined on the set $D_x=\{y\in V^{N\setminus M}: x\cup y\in D\}$.

If $m\in N$, then for each $x\in V$ there exists the unique function $x{:}m\in V^{\{m\}}$ attaining the value $x$ at $m$; formally, $x{:}m=\{(m,x)\}$. In such a setting, we define
for all $x\in V$ and $y\in V^{N\setminus\{m\}}$
\begin{equation*}
(x,y)_m=(x{:}m)\cup y,
\end{equation*}
i.e.\ for any $z=(z_n)_{n\in N}\in V^N$ we have $z=(x,y)_m$ if $z_m=x$ and $z_n=y_n$ for any $n\in N\setminus\{m\}$.
Furthermore, for all $x\in V$, $y\in D_{x{:}m}$, and a function $f$ defined on $D$, we set
\begin{equation*}
f_{x{:}m}(y)=f\big((x,y)_m\big);
\end{equation*}
for simplicity we write $f(x,y)_m$ instead of $f\big((x,y)_m\big)$.

Given all $A\subseteq V$ and $B\subseteq V^{N\setminus\{m\}}$, we set
\begin{equation*}
A\times_mB=\big\{(x,y)_m :x\in A,y\in B\big\}.
\end{equation*}


\section{Multidimensional increment of functions}\label{section3}
For all $x=(x_n)_{n\in N}\in V^N$ and $y=(y_n)_{n\in N}\in V^N$, we define
\begin{equation*}
\Pi(x,y)=\prod_{n\in N}\{x_n,y_n\}.
\end{equation*}
Note that $\Pi$ is an interval operator, which generates a convex structure on $V^N$ in the sense of \cite[Chapter I, \S 4]{V93}.
So, we use the terminology from \cite{V93}. In particular, we say that a set $A\subseteq V^N$ is \emph{$\Pi$\nobreakdash-convex} if $\Pi(x,y)\subseteq A$ for all $x,y\in A$. More generally, if $Z$ is a set, $\mathcal{P}(Z)$ denotes its power set, and $J\colon Z\times Z\to \mathcal P(Z)$ is an operator, then a set $A\subseteq Z$ is called \emph{$J$\nobreakdash-convex} if $J(x,y)\subseteq A$ for all $x,y\in A$. 

Furthermore, for all $x=(x_n)_{n\in N}\in V^N$, $y=(y_n)_{n\in N}\in V^N$, and $M\subseteq N$, we define 
\begin{equation*}
\pi_M(x,y)=x|_M\cup y|_{N\setminus M},
\end{equation*}
i.e. $\pi_M(x,y)=(\pi_{M,n}(x,y))_{n\in N}$, where
$\pi_{M,n}(x,y)=x_n$ if $n\in M$, and $\pi_{M,n}(x,y)=y_n$ if $n\in N\setminus M$.
Therefore, we have 
\begin{equation*}
\Pi(x,y)=\{\pi_M(x,y):M\subseteq N\}.
\end{equation*}

Assume that $X\subseteq V^N$ and let $x,y\in X$ be such that $\Pi(x,y)\subseteq X$. The \emph{multidimensional increment of a function $f\colon X\to\mathbb{R}$ at $(x,y)$} is the number	
\begin{equation}\label{equ:defIncrement}
\square f(x,y)=\sum_{M\subseteq N}(-1)^{|M|}f\big(\pi_M(x,y)\big).
\end{equation}

Note that in the case where $X$ is $\Pi$\nobreakdash-convex, we have $\Pi(x,y)\subseteq X$ for all $x,y\in X$, and hence the multidimensional increment of any function $f\colon X\to\mathbb{R}$ is well-defined at every $(x,y)\in X\times X$. 
Moreover, for $r=1$ we have $\square f(x,y)=f(y)-f(x)$, and hence the multidimensional increment coincides with the one-dimensional increment. We now present an equivalent formula for the multidimensional increment, which allows for an inductive definition analogous to the one-dimensional case.

\begin{proposition}\label{prop:indIncrement}
Assume that $m\in N$ and let $f\colon D\to\mathbb{R}$. Then 
\begin{equation*}\label{equ:indIncrement}
\square f(x,y)=\square f_{y_m{:}m}\big(x|_{N\setminus \{m\}},y|_{N\setminus \{m\}}\big)-\square f_{x_m{:}m}\big(x|_{N\setminus \{m\}},y|_{N\setminus \{m\}}\big)
\end{equation*}
for all $x=(x_n)_{n\in N}\in D$ and $y=(y_n)_{n\in N}\in D$ such that $\Pi(x,y)\subseteq D$.
\end{proposition}

\begin{proof} 
Put $N'=N\setminus\{m\}$. Using \eqref{equ:defIncrement}, we obtain
\begin{equation}\label{equ:indIncrement1}
\begin{split}
\square f(x,y)&=\sum_{M\subseteq N}(-1)^{|M|}f\big(\pi_M(x,y)\big)\\
&=\sum_{M\subseteq N'}(-1)^{|M|}f\big(\pi_M(x,y)\big)
+\sum_{M\subseteq N'}(-1)^{|M\cup\{m\}|}f\big(\pi_{M\cup\{m\}}(x,y)\big)\\ 
&=\sum_{M\subseteq N'}(-1)^{|M|}f\big(\pi_M(x,y)\big)
-\sum_{M\subseteq N'}(-1)^{|M|}f\big(\pi_{M\cup\{m\}}(x,y)\big).
\end{split}
\end{equation}

Fix $M\subseteq N'$. Since $m\notin N'$, we have 
\begin{equation*}
\pi_{M,m}(x,y)=y_m\quad\text{and}\quad \pi_{M\cup\{m\},m}(x,y)=x_m,
\end{equation*}
whereas, putting $x'=x|_{N'}$ and $y'=y|_{N'}$, we have for every $n\in N'$
\begin{equation*}
\pi_{M,n}(x,y)=\pi_{M,n}(x',y')\quad\text{and}
\quad\pi_{M\cup \{m\},n}(x,y)=\pi_{M\cup\{m\},n}(x',y').
\end{equation*}
Therefore,
\begin{equation*}
\pi_M(x,y)=(y_m,\pi_M(x',y')|_{N'})_m\quad\text{and}\quad 
\pi_{M\cup\{m\}}(x,y)=(x_m,\pi_M(x',y')|_{N'})_m.
\end{equation*}

In consequence, returning to \eqref{equ:indIncrement1} and again using \eqref{equ:defIncrement}, we obtain
\begin{align*}
\square f(x,y)&=\sum_{M\subseteq N'}\!(-1)^{|M|}f(y_m,\pi_M(x',y')|_{N'})_m
-\!\!\!\!\!\!\!\!\!\sum_{M\subseteq N\setminus \{m\}}\!\!\!\!\!\!(-1)^{|M|}f(x_m,\pi_M(x',y')|_{N'})_m\\
&=\sum_{M\subseteq N'}(-1)^{|M|}f_{y_m:m}\big(\pi_M(x',y')\big) 
-\sum_{M\subseteq N'}(-1)^{|M|}f_{x_m:m}\big(\pi_M(x',y')\big)\\
&=\square f_{y_m:m}(x',y')-\square f_{x_m:m}(x',y'),
\end{align*}
which completes the proof. 
\end{proof}


\section{Gradient with respect to groups of variables}\label{section4}
Throughout this section, we assume that the set $D$ is open in $V^N$. We also fix $f\colon D\to\mathbb{R}$ and $M=\{m_1,\dots,m_p\}\subseteq N$, assuming that $p$ is a positive integer and $m_1<\dots<m_p$.

For any $l=(n,k)\in L$ we denote the $l$\nobreakdash-th partial derivative of $f$ at $x\in D$ by $\partial_lf(x)$. We say that $f$ is
\emph{separately differentiable} if for all $x\in D$ and $l\in L$ the $l$\nobreakdash-th partial derivative $\partial_lf(x)$ exists. We denote the \emph{gradient} of $f$ at $x\in D$ by
\begin{equation*}
\nabla f(x)=\big(\partial_lf(x)\big)_{l\in L},
\end{equation*}
and so $\nabla f(x)u=\sum_{l\in L}\partial_lf(x)u_l$ for every $u=(u_l)_{l\in L}\in V^N$.
Next, for any $n\in N$, we define the \emph{$n$\nobreakdash-th partial gradient} of $f$ at $x\in D$ by
\begin{equation*}\label{nderel}
\nabla\!_n f(x)=\big(\partial_{(n,k)}f(x)\big)_{k\in K},
\end{equation*}
and set $\nabla\!_n f(x)v=\sum_{k\in K}\partial_{(n,k)} f(x)v_k$ for every $v=(v_k)_{k\in K}\in V$.

Let $\lambda=(\lambda_k)_{k\in K^M}\in\mathbb{R}^{K^M}$ be an indexed family. For any $u\in V^M$, we define
\begin{equation}\label{equ:notation_lambda_xN}
\lambda u ^M =\sum_{k\in K^M}\lambda_k \odot^M_k(u),\quad\text{where } 
\odot_k^M(u)=\prod_{m\in M}u_{m,k_m}.
\end{equation}
Obviously, $\lambda u^M$  is an $|M|$\nobreakdash-linear functional on $V^M$.

Now, for every $u\in V^M$, we put $\|u\|^M =\prod_{m\in M}\|u_{m}\|$. We show that the following analogue of the Cauchy–Bunyakovsky–Schwarz inequality holds:
\begin{equation}\label{equ:CBSinequality}
\left|\lambda u^M \right|\le \|\lambda\|\cdot \|u\|^M.
\end{equation}

Observe first that
\begin{equation*}
\sum_{k\in K^M}\prod_{m\in M}u_{m,k_m}^2=\prod_{m\in M}\sum_{k\in K}u_{m,k}^2.  
\end{equation*}
This together with the classical Cauchy–Bunyakovsky–Schwarz inequality gives
\begin{align*}
\left|\lambda u^M \right|^2
&\!\leq\sum_{k\in K^M} \lambda_k^2\sum_{k\in K^M}\!\big(\!\odot^M_k(u)\big)^2
=\|\lambda\|^2\sum_{k\in K^M}\prod_{m\in M}u_{m,k_m}^2
=\|\lambda\|^2\prod_{m\in M}\sum_{k\in K}u_{m,k} ^2\\
&=\|\lambda\|^2\prod_{m\in M}\|u_{m} \|^2
=\|\lambda\|^2 \left(\|u\|^M\right)^2.
\end{align*}

Let us pass to the next definition.
For any $k=(k_m)_{m\in M}\in K^M$, we put $\widetilde\partial_kf=\partial_{(m_1, k_{m_1})}
\cdots\partial_{(m_p, k_{m_p})}f$. Then we define the \emph{$M$\nobreakdash-gradient} of~$f$ at $x\in D$ by
\begin{equation*}
\nabla\!_Mf(x)=\big(\widetilde\partial_kf(x)\big)_{k\in K^M}.
\end{equation*}
Observe that using the notation established in \eqref{equ:notation_lambda_xN}, we have
\begin{equation*}
\nabla\!_Mf(x)u^M =\sum_{k\in K^M}\widetilde{\partial}_kf(x)\odot^M_k(u)
\end{equation*}
for every $u\in V^M$.
In the case $s=1$, identifying $(m_k,1)$ with $m_k$, we have 
$\nabla\!_Mf(x)=\partial_Mf(p)=\partial_{m_1}
\cdots \partial_{m_p}f(x)$ and $\nabla\!_M f(x)u^M=\partial_Mf(x)u_{m_1}
\dots u_{m_p}$ for every $x\in D$.


\section{Multidimensional analogue of Lipschitz functions}\label{section5}
Let $D\subseteq V^N$ be a closed subset.
We say $f\colon D\to\mathbb R$ is a \emph{$\mathrm{C}^N$\nobreakdash-function} if there exists an extension $\widetilde{f}\colon V^N\to\mathbb{R}$ of $f$ such that for every $M\subseteq N$  the gradient~$\nabla\!_{M}\widetilde{f}$ exists at each point in~$V^N$ and is continuous on~$V^N$, i.e.\ each partial derivative $\widetilde{\partial}_k\widetilde{f}$ is continuous for every $k\in K^M$.

Note that $f\in C^r(V^N,\mathbb R)$ implies that $f$ is a $\mathrm{C}^N$\nobreakdash-function; however, the converse implication does not hold.
Observe also that for every $M\subseteq N$ the gradient $\nabla\!_M\widetilde{f}$ may depend on the extension~$\widetilde{f}$.

It is possible to show that in the definition of $\mathrm{C}^N$\nobreakdash-functions we only need to extend $f$ to an open superset of $D$.  For $r=1$, this follows from Whitney’s Extension Theorem (see \cite[Theorem 6.5.1]{EvGa}).

We say that $f$ is \emph{$N$\nobreakdash-dimensional Lipschitz} if there exists a constant $C\ge 0$ such that 
\begin{equation*}
\left|\square f(x,y)\right|\le C\|x-y\|^N\quad\text{for all }x,y\in X
\text{ such that }\Pi(x,y)\subseteq D.
\end{equation*}

Note that in the case where $r=1$, the class of $N$\nobreakdash-dimensional Lipschitz functions coincides with the usual class of Lipschitz functions. Observe also that an $N$\nobreakdash-dimensional Lipschitz function need not even be continuous.

We now prove the following version of the mean value theorem.

\begin{proposition}\label{prop:finIncrementNdim}
If $f\colon V^N\to \mathbb{R}$ is a $\mathrm{C}^N$\nobreakdash-function, then for all $a,b\in V^N$ there exists $c\in P(a,b)$ such that $\square f(a,b)=\nabla\!_Nf(c)(b-a)^N$.
\end{proposition} 

\begin{proof} 
We begin with the case $r=1$. We define the function $g\colon[0,1]\to\mathbb{R}$ by $g(t)=f\left(t(b-a)+a\right)$. 
Then, by the mean value theorem (also known as the Lagrange finite-increment formula) there exists $\tau\in[0,1]$ such that $g(1)-g(0)=g'(\tau)$. 
Putting $c=\tau(b-a)+a$, we see that $c\in [a,b]$, and then
\begin{equation*}
\square f(a,b)=f(b)-f(a)=g(1)-g(0)=g'(\tau)=\nabla f(c)(b-a)=\nabla_N f(c)(b-a)^N.
\end{equation*} 

Suppose now that the claim of the proposition holds for every $p\in\{1,\ldots,r-1\}$ and the set $N_p=\{1,\ldots,p\}\subseteq N$. 
Fix a $\mathrm{C}^{N}$\nobreakdash-function $f\colon V^{N}\to \mathbb{R}$, 
$a=(a_n)_{n\in N}\in V^{N}$ and $b=(b_n)_{n\in N}\in V^{N}$. Put $N'=N_{r-1}$,
$a'=(a_n)_{n\in N'}\in V^{N'}$, and $b'=(b_n)_{n\in N'}\in V^{N'}$. 
Moreover, we easily conclude that the formula $g=f_{b_{r}{:}r}-f_{a_{r}{:}r}$ defines a $\mathrm{C}^{N'}$\nobreakdash-function. Then, by \cref{prop:indIncrement} and the inductive hypothesis, we obtain the existence of $c'\in P(a',b')$ such that 
\begin{align*}
\square f(a,b)&=\square g(a',b')=\nabla\!_{N'}g(c')(b'-a')^{N'}\\
&=\nabla\!_{N'}f_{b_{r}{:}r}(c')(b'-a')^{N'}
-\nabla\!_{N'}f_{a_{r}{:}r}(c')(b'-a')^{N'}.
\end{align*}
Define $h\colon V\to\mathbb{R}$ by
$
h(x)=\nabla\!_{N'}f_{x{:}r}(c')(b'-a')^{N'}
$
and note that $h$ is differentiable since $f$ is a $\mathrm{C}^{N}$\nobreakdash-function.
Again, by the inductive hypothesis, we obtain the existence of $c_{r}\in [a_{r},b_{r}]\subseteq V$ such that $\square f(a,b)=\square h(a_r,b_r)=\nabla h(c_r)(b_r-a_r)$. Finally, putting  $c=(c_r,c')_r$, we obtain 
\begin{align*}
\square f(a,b)&=\nabla h(c_r)(b_r-a_r)
=\sum_{k\in K}\frac{\partial h}{\partial x_k}(c_r)(b_{r,k}-a_{r,k})\\
&=\sum_{k\in K}\frac{\partial}{\partial x_k} \big(\nabla\!_{N'}f_{x{:}r} (c')(b'-a')^{N'}\big)\big|_{x=c_r}(b_{r,k}-a_{r,k})\\
&=\sum_{k\in K}\left.\frac{\partial}{\partial x_k}\left(\sum_{k'\in K^{N'}}
\widetilde{\partial}_{k'}f(x,c')_r\odot^{N'}_{k'}(b'-a')\right)
\right|_{x=c_r}(b_{r,k}-a_{r,k})\\
&=\sum_{k\in K}\sum_{k'\in K^{N'}}\partial_{(r,k)}\widetilde{\partial}_{k'}
f(c)(b_{r,k}-a_{r,k})\odot^{N'}_{k'}(b'-a')\\
&=\sum_{(k,k')_r\in K^{N}}\widetilde{\partial}_{(k,k')_r}f(c)\odot^N_{(k,k')_r}(b-a)
=\sum_{\ell\in K^{N}}\widetilde{\partial}_{\ell}f(c)\odot^N_{\ell}(b-a)\\
&=\nabla\!_{N}f(c)(b-a)^{N},
\end{align*}
and the proof is complete.
\end{proof}

\begin{proposition}\label{prop:NdimLipschitzOfCNfunction}
Assume that $f\colon V^N\to \mathbb{R}$ is a $\mathrm{C}^N$\nobreakdash-function such that $\nabla\!_Nf$ is bounded on a $P$\nobreakdash-convex set $X\subseteq V^N$. Then $f$ is $N$\nobreakdash-dimensional Lipschitz on~$X$. 
\end{proposition} 

\begin{proof}
Since $\nabla\!_Nf$ is bounded on $X$, the number $C=\sup_{x\in X}\big\|\nabla\!_Nf(x)\big\|$ is finite. Fix $a,b\in X$.
\Cref{prop:finIncrementNdim} together with the $P$\nobreakdash-convexity of $X$ implies that there exist $z\in P(a,b)\subseteq X$ such that
$\square f(a,b)=\nabla\!_Nf(z)(b-a)^N$. Finally, by \eqref{equ:CBSinequality}, we get
\begin{equation*}
\big|\square f(a,b)\big|
=\big|\nabla\!_Nf(z)(b-a)^N\big|
\le \big\|\nabla\!_Nf(z)\big\|
\cdot \|b-a\|^N
\le C\|b-a\|^N,
\end{equation*}
and the proof is complete.
\end{proof}
	

\section{The gradient increment and the average gradient of multidimensional Lipschitz functions}\label{section6}
Assume that $f\colon V^N\to\mathbb{R}$ is a $\mathrm{C}^N$\nobreakdash-function. For any compact set $T\subseteq V^N$, we define the \emph{$N$\nobreakdash-gradient increment} of $f$ on $T$ by 
\begin{equation*}
\widetilde{\nabla}\!_Nf(T)=\int_T \nabla\!_N f(x)\, d\mu(x)
=\left(\int_T \widetilde{\partial}_{k}f(x)\, d\mu(x)\right)_{k\in K^N},
\end{equation*}
and for any $u\in V^N$, we put
\begin{equation*}
\widetilde{\nabla}\!_Nf(T)u=\int_T \nabla\!_N f(x)u\,d\mu(x)
=\sum_{k\in K^N}\odot^N_k(u)\int_T\widetilde{\partial}_{k}f(x)\, d\mu(x).
\end{equation*}
Then, assuming that $T$ is of positive Lebesgue measure, we define the \emph{average $N$\nobreakdash-gradient of $f$ on~$T$} by 
\begin{equation*}
\overline{\nabla}\!_Nf(T)=\tfrac{1}{\mu(T)}\widetilde{\nabla}\!_Nf(T)
\end{equation*}
and $\overline{\nabla}\!_Nf(T)u=\tfrac{1}{\mu(T)}\widetilde{\nabla}\!_Nf(T)u$ for every $u\in\mathbb{R}^K$.

We conclude this section with a result that shows an equality reminiscent of the fundamental theorem of calculus. Since the result will not be used in this paper, we omit the proof.

\begin{fact}
Assume that $f\colon V^N\to\mathbb{R}$ is a $\mathrm{C}^N$\nobreakdash-function. Let $a=(a_n)_{n\in N}\in\mathbb{R}^N$ and $b=(b_n)_{n\in N}\in\mathbb{R}^N$. If $a_n\le b_n$ for every $n\in N$, then
\begin{equation*}
\widetilde{\nabla}_Nf(P(a,b))=\int_{P(a,b)}\partial_Nf\, d\mu=\square f(a,b).
\end{equation*}
\end{fact} 


\section{Multidimensional MW-operators}\label{section7}
For any $x\in D$, we put
\begin{equation*}
\Gamma(x)=\big\{\gamma^i(x):i\in I\big\}\quad\text{and}\quad 
\Gamma_\pi(x)=\bigcup_{i\in I}\Pi(\gamma^i(0),\gamma^i(x)).
\end{equation*}

\begin{definition}[admissible]\label{defAmisibleSet}
A set $X\subseteq D$ is said to be \emph{admissible} if $0\in X$ and $\Pi(\gamma^i(0),\gamma^i(x))\subseteq X$ whenever $x\in X$ and $i\in I$. 	
\end{definition}

Put $\Gamma_\pi^0(0)=\{0\}$ and $\Gamma_\pi^{k+1}(0)=\Gamma_\pi(\Gamma_\pi^k(0))=\bigcup_{x\in \Gamma_\pi^k(0)}\Gamma_\pi(x)$ for every $k\in\mathbb{N}_0$. It is easy to see that the set~$\textbf{S}$ defined as follows 
\begin{equation*}
\textbf{S}=\bigcup_{k\in\mathbb{N}_0}\Gamma_\pi^k(0)
\end{equation*}
is the minimal (in the sense of inclusion) admissible set. If the family $\gamma$ consists of continuous functions, then it is easy to see that the minimal closed admissible set, is the closure of $\textbf{S}$; let us denote this set by $\textbf{T}$, i.e. 
\begin{equation*}
\textbf{T}=\overline{\textbf{S}}.	
\end{equation*}

Given an admissible set $X\subseteq D$, we put
\begin{equation*}
\mathcal{F}(X)=\big\{f\in\mathbb{R}^X:\text{$\sum_{i\in I}\square f\big(\gamma^i(0),\gamma^i(x)\big)$ converges for every $x\in X$}\big\}.
\end{equation*}

It is easy to see that $\mathcal{F}(X)$ is a vector subspace of $\mathbb{R}^X$.

\begin{definition}[multidimensional MW\nobreakdash-operator]\label{defMWO}
The linear operator $\mathbb{M}\colon \mathcal{F}(X)\to \mathbb{R}^X$ given by
\begin{equation*}\label{equ:MWoperator}
\mathbb{M}f(x)=\sum_{i\in I}\square f\big(\gamma^i(0),\gamma^i(x)\big)\quad\text{for every } x\in X 
\end{equation*} 
is said to be the \emph{multidimensional MW\nobreakdash-operator} (briefly, the \emph{MW-operator}). 
\end{definition}

The formula for the MW\nobreakdash-operator in \cref{defMWO} is given in terms of the multidimensional increment~$\square$. Next, we provide an equivalent formula for this operator; however, to do so, we need to introduce one additional notation.

For any $i\in I$, $x\in X$, and $M\subseteq N$, we define
\begin{equation}\label{gammaiM}
\gamma^i_M(x)=\pi_M\big(\gamma^i(0),\gamma^i(x)\big),
\end{equation}
i.e. $\gamma^i_M(x)=\big(\gamma^i_{M,n}(x)\big)_{n\in N}$, where 
\begin{equation*}
\gamma^i_{M,n}(x)=\pi_{M,n}\big(\gamma^i(0),\gamma^i(x)\big)=
\left\{
\begin{array}{cl}
\gamma^i_n(0),&\text{if }n\in M,\\
\gamma^i_n(x),&\text{if }n\in N\setminus M.
\end{array}
\right.
\end{equation*} 
Thus, using \eqref{equ:defIncrement}, we obtain the following equivalent formula for the MW-operator 
\begin{equation}\label{equ:MWoperatorDirect}
\mathbb{M}f(x)=\sum_{i\in I}\sum_{M\subseteq N}(-1)^{|M|}f\big(\gamma^i_M(x)\big).
\end{equation}

\begin{definition}[multidimensional MW\nobreakdash-functional equation]\label{defMWE}
The functional equation
\begin{equation*}\label{equ:MWE}
f=\mathbb{M}f  
\end{equation*}
is called the \emph{multidimensional MW\nobreakdash-functional equation} (briefly, the \emph{MW\nobreakdash-equation}). The unknown function $f$ is assumed to belong to $\mathcal{F}(X)$.
\end{definition}

To understand the concepts introduced and the notation adopted, we provide four examples.

\begin{example}\label{M11}
Assume that $r=s=1$, i.e. $D=V^N=V=\mathbb{R}$. Let $I=\{0,1\}$ and 
\begin{equation*}
\gamma^i(x)=\tfrac{x+i}{2}\quad \text{for all }i\in I\text{ and }x\in\mathbb{R}.
\end{equation*}
It is clear that the minimal admissible sets $\textbf{S}$ and $\textbf{T}$ are given by
\begin{equation*}
\textbf{S}=\big\{\tfrac{k}{2^n}:n\in\mathbb{N},k=0,1,\dots,2^n-1\big\}\quad\text{and}\quad 
\textbf{T}=[0,1].
\end{equation*}
By choosing the admissible set $X=\textbf{T}$, we obtain the MW\nobreakdash-equation~\eqref{MWeq} and the corresponding MW\nobreakdash-operator $\mathbb{M}$ given by~\eqref{MWop}.
\end{example}

The next two examples show that the two-dimensional MW\nobreakdash-operator may take different forms on the same admissible set, depending on the sets $N$ and $K$.

\begin{example}\label{M12}
Assume that $r=1$ and $s=2$, i.e. $D=V^N=V=\mathbb{R}^2$. Let $I=\{0,1\}^2$ and
\begin{equation*}\label{gamma}
\gamma^{(i,j)}(x,y)=\big(\tfrac{x+i}{2},\tfrac{y+j}{2}\big)\quad\text{for all }(i,j)\in I\text{ and }(x,y)\in\mathbb{R}^2.
\end{equation*}
It this case, the minimal admissible sets $\textbf{S}$ and $\textbf{T}$ are given by
\begin{equation*}\label{SgammaTgamma}
\textbf{S}=\big\{\tfrac{k}{2^n}:n\in\mathbb{N},k=0,1,\dots,2^n-1\big\}^2\quad\text{and}\quad 
\textbf{T}=[0,1]^2.
\end{equation*}
Choosing an admissible set $X$ and using \eqref{equ:MWoperatorDirect}, we conclude that the two-dimensional MW\nobreakdash-operator in the considered case takes the form
\begin{equation*}\label{equ:MWoperator12}
\mathbb{M}f(x,y)=\sum_{(i,j)\in I} \left[f\big(\tfrac{x+i}{2},\tfrac{y+j}{2}\big)-f\big(\tfrac{i}{2},\tfrac{j}{2}\big)\right]
\end{equation*}  
for every $(x,y)\in X$.
\end{example}

\begin{example}\label{M21}
Assume that $r=2$ and $s=1$, i.e. $V=\mathbb{R}$ and $D=V^N=\mathbb{R}^2$. Let $\gamma$ be the same as in \cref{M12}. Then the minimal admissible sets $\textbf{S}$ and $\textbf{T}$ are also the same as in \cref{M12}. Choosing an admissible set $X$ and again making use of~\eqref{equ:MWoperatorDirect}, we deduce that in this setting the two-dimensional MW\nobreakdash-operator takes the form
\begin{equation*}\label{equ:MWoperator21}
\mathbb{M}f(x,y)=\sum_{(i,j)\in I} \left[f\big(\tfrac{x+i}{2},\tfrac{y+j}{2}\big)-f\big(\tfrac{i}{2},\tfrac{y+j}{2}\big)
-f\big(\tfrac{x+i}{2},\tfrac{j}{2}\big)+f\big(\tfrac{i}{2},\tfrac{j}{2}\big)
\right]
\end{equation*}  
for every $(x,y)\in X$.
\end{example}

The fourth example generalizes all the preceding ones.

\begin{example}\label{Mrs}
Let $I=\{0,1\}^L$ and 
\begin{equation*}
\gamma^i(x)=\big(\tfrac{x_\ell+i_\ell}{2}\big)_{\ell\in L}
\quad\text{for all }i\in I\text{ and }x\in\mathbb{R}^L.
\end{equation*}
Now that the minimal admissible sets $\textbf{S}$ and $\textbf{T}$ are given by 
\begin{equation*}
\textbf{S}=\big\{\tfrac{k}{2^n}:n\in\mathbb{N},k=0,1,\dots,2^n-1\big\}^L\quad\text{and}\quad 
\textbf{T}=[0,1]^L.
\end{equation*}
Applying \eqref{equ:MWoperatorDirect} to an admissible set $X$, we arrive at the multidimensional MW\nobreakdash-operator in the form
\begin{equation*}
\mathbb{M}f(x)=\sum_{i\in I}\sum_{M\subseteq N}(-1)^{|M|}
f\Big(\big(\tfrac{i_\ell}{2}\big)_{\ell\in M\times K}\cup
\big(\tfrac{x_\ell+i_\ell}{2}\big)_{\ell\in (N\setminus M)\times K}\Big)
\end{equation*}
for every $x\in X$.
\end{example}

Our main result, formulated in \cref{section9} (see \cref{thm:MWequatoC1}), implies that a function $f\in C^r([0,1]^L,\mathbb{R})$ is a solution of the MW\nobreakdash-equation corresponding to the MW\nobreakdash-operator from \cref{Mrs} if and only if there exists $\lambda\in \mathbb{R}^{K^N}$ such that
\begin{equation*}
f(x)=\sum_{k\in K^N}\lambda_k\prod_{n\in N}x_{n,k_n}\quad\text{for every }x\in [0,1]^L.
\end{equation*}
In particular, for the cases listed in \cref{M11,M12,M21}, we obtain:
\begin{itemize}
\item[(i)] $f\in C^1([0,1],\mathbb R)$ is a solution of the MW\nobreakdash-equation \eqref{MWeq} (cf.~\cref{M11}) if and only if there exists $\lambda\in \mathbb{R}$ such that \begin{equation*}
f(x)=\lambda x\quad\text{for every }x\in[0,1],
\end{equation*}
\item[(ii)] $f\in C^1([0,1]^2,\mathbb R)$ is a solution of the MW\nobreakdash-equation corresponding to the MW\nobreakdash-operator from \cref{M12} if and only if there exist $\lambda_1,\lambda_2\in\mathbb R$ such that 
\begin{equation*}
f(x_1,x_2)=\lambda_1 x_1+\lambda_2 x_2\quad\text{for every }(x,y)\in[0,1]^2,
\end{equation*}
\item[(iii)] $f\in C^2([0,1]^2,\mathbb R)$ is a solution of the MW\nobreakdash-equation corresponding to the MW\nobreakdash-operator from \cref{M21} if and only if there exists $\lambda\in\mathbb R$ such that 
\begin{equation*}
f(x_1,x_2)=\lambda x_1x_2\quad\text{for every }(x,y)\in[0,1]^2.
\end{equation*}
\end{itemize} 

It is known that in (i) the class $C^1([0,1],\mathbb R)$ can be replaced by the class of absolutely continuous functions (see, e.g., \cite{R57,LP77}).


\section{Iterations of the MW-operators}\label{section8}
Throughout this section, we fix an admissible set $X\subseteq D\subseteq V^N$. 

For all $p\in\mathbb{N}$ and multi-indices $i=(i_1,\ldots,i_p)\in I^p$ we define $\gamma^i\colon D\to D$ by
\begin{equation*}
\gamma^i=\gamma^{i_1}\circ\gamma^{i_2}\circ\dots\circ \gamma^{i_p}.
\end{equation*}

As usual, for any $p\in\mathbb N$ we denote the $p$\nobreakdash-iterate of the MW\nobreakdash-operator $\mathbb{M}$ by $\mathbb{M}^p$. We define the domains of the iterates by recursion, setting $\mathcal{F}^1(X)=\mathcal{F}(X)$ and 
\begin{equation*}
\mathcal{F}^{p}(X)=\big\{f\in \mathcal{F}(X):\,\mathbb{M}^k f\in \mathcal{F}(X) \text{ for every }k\in \{1,\ldots,p-1\}\big\}
\end{equation*} 
for every $p\geq 2$. We also put
\begin{equation*}
\mathcal{F}^\infty(X)=\bigcap_{p=1}^\infty \mathcal{F}^p(X).
\end{equation*}
Thus, $\mathcal{F}^p(X)$ is the domain of $\mathbb M^p$, and $\mathcal{F}^\infty(X)$ is the common domain of all iterates. Obviously, $\mathcal{F}^\infty(X)$ is a vector subspace of $\mathbb{R}^X$, and $\mathcal{F}^\infty(X)=\mathcal{F}(X)=\mathbb{R}^X$ in the case where $I$ is a finite set.

\begin{proposition}\label{prop:IterationsMWoperatorsSpecial}
Assume that for any $i\in I$, the function $\gamma^i$ is of the form
\begin{equation*}
\gamma^i(x)=\big(\gamma^i_n(x_n)\big)_{n\in N}\quad\text{for every }x=(x_n)_{n\in N}\in D,
\end{equation*}	
where $\gamma^i_n\colon D_n\to D_n$ and $D_n=\{x_n\in V:x\in D\}$ for every $n\in N$.
Then for all $p\in\mathbb{N}$ and $f\in \mathcal{F}^p(X)$ we have 
\begin{equation*}\label{equ:mMWoperatorSpecial}
\mathbb{M}^p f(x)=\sum_{i\in I^p}\square f\big(\gamma^i(0),\gamma^i(x)\big),
\end{equation*}
or equivalently, extending the  definition of $\gamma^i_M$ given in \eqref{gammaiM} to all $i \in I^p$, $x \in X$, and $M \subseteq N$, 
\begin{equation}\label{equ:mMWoperatorSpecialDirect}
\mathbb{M}^p f(x)=\sum_{i\in I^p}\sum_{M\subseteq N}(-1)^{|M|}f\big(\gamma^i_M(x)\big).
\end{equation}
\end{proposition}	

\begin{proof}
We proceed by induction on $p$ to prove \eqref{equ:mMWoperatorSpecialDirect}.
For $p=1$, \eqref{equ:mMWoperatorSpecialDirect} reduces to~\eqref{equ:MWoperatorDirect}. Fix $p\in\mathbb N$ and suppose that \eqref{equ:mMWoperatorSpecialDirect} holds for every $f\in \mathcal{F}^p(X)$.

Let us fix $f\in\mathcal{F}^{p+1}(X)$ and $x\in X$. Then $\mathbb{M}f\in\mathcal{F}^p(X)$. Putting 
\begin{equation*}
\Sigma_M^{i,j}(x)=\sum_{L\subseteq N}(-1)^{|L|}f\big(\gamma^i_M(\gamma^j_L(x))\big)
\end{equation*}
for all $i\in I$, $j\in I^p$ and $M\subseteq N$, we have
\begin{equation}\label{equ:MWspecial(p+1)}
\begin{split}
\mathbb{M}^{p+1} f(x)&=\mathbb{M}^p\big(\mathbb{M} f\big)(x)
=\sum_{j\in I^p}\sum_{L\subseteq N}(-1)^{|L|}\mathbb{M}f\big(\gamma^j_L(x)\big)\\
&=\sum_{j\in I^p}\sum_{L\subseteq N}(-1)^{|L|}
\sum_{i\in I}\sum_{M\subseteq N}(-1)^{|M|}f\big(\gamma^i_M(\gamma^j_L(x))\big)\\
&=\sum_{j\in I^p}\sum_{i\in I}\sum_{M\subseteq N}(-1)^{|M|}\Sigma_M^{i,j}(x).
\end{split}
\end{equation}

Let $M\subseteq N$ and $L=L'\cup L''$ with $L'\subseteq M$ and $L''\subseteq N\setminus M$. Then
\begin{equation}\label{equ:gammaKLij}
\begin{split}
\gamma^i_{M,n}(\gamma^j_L(x))
&=\begin{cases}
\gamma^i_n(0),&\text{if }n\in M,\\
\gamma^i_n\big(\gamma^j_{L,n}(x_n)\big),&\text{if }n\in  N\setminus M
\end{cases}\\
&=\begin{cases}
\gamma^i_n(0),&\text{if }n\in M,\\
\gamma^i_n\big(\gamma^j_n(0)\big),&\text{if }n\in L'',\\
\gamma^i_n\big(\gamma^j_n(x_n)\big),&\text{if }n\in (N\setminus M)\setminus L''.
\end{cases}
\end{split}
\end{equation}
Therefore, $\gamma^i_M\big(\gamma^j_L(x)\big)$ does not depend on~$L'$ and
$\gamma^i_M\big(\gamma^j_L(x)\big)=\gamma^i_M\big(\gamma^j_{L''}(x)\big)$.
In consequence, 
\begin{equation}\label{Sigma}
\begin{split}
\Sigma_M^{i,j}(x)&=\sum_{L'\subseteq M}(-1)^{|L'|}\sum_{L''\subseteq N\setminus M}(-1)^{|L''|}
f\big(\gamma^i_M(\gamma^j_{L'\cup L''}(x))\big)\\
&=\sum_{L'\subseteq M}(-1)^{|L'|}\sum_{L''\subseteq N\setminus M}
(-1)^{|L''|}f\big(\gamma^i_M(\gamma^j_{ L''}(x))\big).	
\end{split}		
\end{equation}

Observe now that
\begin{equation}\label{sumL'K}
\sum_{L'\subseteq M}(-1)^{|L'|}=
\begin{cases}
0,&\text{if }M\neq\emptyset,\\
1,&\text{if }M=\emptyset. 
\end{cases}
\end{equation}

For any $i\in I$ and $j=(j_k)_{k=1}^p\in I^p$ we put $i\smallsmile j=(i,j_1,\dots,j_p)\in I^{p+1}$.
Applying~\eqref{equ:gammaKLij} we obtain 
\begin{equation*}
\gamma^i_{\emptyset,n}(\gamma^j_L(x))
=\begin{cases}
\gamma^i_n\big(\gamma^j_n(0)\big),&\text{if }n\in L,\\
\gamma^i_n\big(\gamma^j_n(x_n)\big),&\text{if }n\in N\setminus L
\end{cases}
=\gamma^{i\smallsmile j}_{L,n}(x_n)
\end{equation*}
for every $n\in N$, and hence $\gamma^i_{\emptyset}(\gamma^j_L(x))=\gamma^{i\smallsmile j}_{L}(x)$.
This together with \eqref{sumL'K} lets us continue the computation started in \eqref{Sigma} to get
\begin{equation*}
\Sigma^{i,j}_M(x)=
\begin{cases}
0,&\text{if }M\ne\emptyset,\\ 			
\sum_{L\subseteq N}(-1)^{|L|}f\big(\gamma^{i\smallsmile j}_L(x)\big)&\text{if }M=\emptyset.
\end{cases}
\end{equation*}
Finally, returning to \eqref{equ:MWspecial(p+1)} and substituting $i\smallsmile j$ by $\ell$, we conclude that
\begin{align*}
\mathbb{M}^{p+1} f(x)&=\sum_{i\in I,j\in  I^{p}}\Sigma_\emptyset^{i,j}(x)
=\sum_{i\in I, j\in  I^{p}}\sum_{L\subseteq N}(-1)^{|L|}f\big(\gamma^{i\smallsmile j}_L(x)\big)\\
&=\sum_{\ell\in I^{p+1}}\sum_{L\subseteq N}(-1)^{|L|}f\big(\gamma^{\ell}_L(x)\big),
\end{align*}
and the proof is complete.
\end{proof}


\section{The case of affine functions}\label{section9}
From now on, we assume that $D=V^N$ and the following hypotheses:
\begin{enumerate}[label={\rm (H$_1$)}]
\item\label{H1} For every $i\in I$ there exist $\alpha^i=(\alpha^i_n)_{n\in N}\in(0,1)^N$ and $a^i=(a^i_n)_{n\in N}\in V^N$ such that $q=\sup_{i\in I,n\in N}\alpha^i_n<1$ and $\gamma^i$ is of the form
\begin{equation*}
\gamma^i(x)=\big(\gamma^i_n(x_n)\big)_{n\in N}=\big(\alpha^i_nx_n+a^i_n\big)_{n\in N}
\quad\text{for every }x=(x_n)_{n\in N}\in V^N.
\end{equation*}
\end{enumerate}
\begin{enumerate}[label={\rm (H$_2$)}]
\item\label{H2} There exists a compact admissible set $T\subseteq V^N$ with $\mu(T)>0$ such that
\begin{equation*}
T\meq{\bigsqcup_{i\in I}}^\mu \gamma^i(T)\quad\text{and}\quad
T=\overline{\bigcup_{i\in I}\gamma^i(T)}.
\end{equation*}
\end{enumerate}

For any $i\in I$, we put
\begin{equation*}
\beta^i=\prod\limits_{n\in N}\alpha^i_n
\end{equation*}
and observe that $\mu(\gamma^i(T))=\beta^i\mu(T)$. Then, by \ref{H2}, we obtain
\begin{equation*}
\mu(T)=\mu\left({\bigsqcup\limits_{i\in I}}^\mu \gamma^i(T)\right)
=\sum_{i\in I}\mu(\gamma^i(T))=\sum_{i\in I}\beta^i\mu(T)=\mu(T)\sum_{i\in I}\beta^i.
\end{equation*}
Thus,
\begin{equation}\label{equ:Acond5}
\sum_{i\in I}\beta^i=1.
\end{equation}

Given a function  $g\colon D\to\mathbb{R}^N$ and $\delta>0$, we put
\begin{equation*}
\omega_T(g,\delta)=\sup\left\{\|g(x)-g(y)\|:\,
x\in T,y\in D\text{ such that  }\|x-y\|<\delta\right\}.
\end{equation*}
This notion is closely related to the modulus of continuity. It is not difficult to show that if $g$ is continuous at each point of the compact set~$T$ then $\omega_T(g,\delta)\to 0$ as $\delta\to 0$.

Given a $\mathrm{C}^N$\nobreakdash-function $\widetilde f\colon V^N\to \mathbb{R}$, we define the operator $\mathbb L$ by setting
\begin{equation*}
\mathbb{L}\widetilde{f}(x)=\overline{\nabla}\!_N\widetilde f(T)x^N
\quad\text{for every }x\in V^N.
\end{equation*}

We are now ready to state the main result of this paper.

\begin{theorem}\label{thm:MWequatoC1}
Assume \ref{H1} and \ref{H2}.
Let $X\subseteq V^N$ be an admissible set, $f\colon X\to \mathbb{R}$ be a $\mathrm{C}^N$\nobreakdash-function, and $\widetilde f\colon V^N\to \mathbb{R}$ be one of its $\mathrm{C}^N$\nobreakdash-extensions. Then: 
\begin{enumerate}[label=$(\roman*)$]
\item\label{mainI} $T$ is the minimal closed admissible set, i.e. $T=\mathbf{T}$; 
\item\label{mainII} $\mathcal{F}^\infty(X)$ contains every function that is $N$\nobreakdash-dimensional Lipschitz on every compact subset of~$X$, in particular $f\in \mathcal{F}^\infty(X)$;
\item\label{mainIII} $\lim_{p\to\infty}\mathbb{M}^pf(x)= \mathbb{L}\widetilde{f} (x)$ for every $x\in X$, and moreover, if $X$ is compact, then the convergence is uniform;
\item\label{mainIV} $\mathbb{M}f=f$ if and only if there exists $\lambda\in\mathbb{R}^{K^N}$ such that  
\begin{equation*}
f(x)=\lambda x^N=\sum_{k\in K^N}\lambda_k\odot^N_k(x)\quad\text{for every }x\in X,
\end{equation*}
 i.e.\ $f$ is an $r$\nobreakdash-linear functional.  
\end{enumerate}
\end{theorem}

\begin{proof}
We start with some notations and obvious observations. 

Fix $p\in\mathbb{N}$ and let $\delta=\text{diam}\,T$. In contrast to the setting described at the beginning of this section, we consider multi-indices $i=(i_k)_{k=1}^p\in I^p$. Then, for any $n\in N$, ${\gamma^i_n=\gamma^{i_1}_n\circ\gamma^{i_2}_n\circ\dots\circ\gamma^{i_p}_n}$ is an affine function such that $\gamma^i_n(x)=\alpha^i_nx+a^i_n$ for every $x\in V$, where $a^i_n=\gamma^i_n(0)$ and $\alpha^i_n=\prod_{k=1}^p\alpha^{i_k}_n$. 
For the empty multi-index $\varnothing\in I^0$, we put $\alpha^\varnothing_n=1$, $a^\varnothing_n=0$, and $\gamma^\varnothing_n(x)=x$ for every $x\in V$. 

For any $i=(i_k)_{k=1}^p\in I^p$ we set $\gamma^i=(\gamma^i_n)_{n\in N}$, $T^i=\gamma^i(T)$, $a^i=(a^i_n)_{n\in N}$, and $\beta^i=\prod_{n\in N}\alpha^i_n$. Therefore, $\mu(T^i)=\mu(T)\beta^i$ and $\beta^i=\prod_{k=1}^p\beta^{i_k}$. So, using~\eqref{equ:Acond5}, we have
\begin{equation}\label{equ:Acondition8}
\sum_{i\in I^p}\beta^i=\prod_{k=1}^p\sum_{i_k\in I}\beta^{i_k}=1.
\end{equation}

Proceeding by induction and using \ref{H2}, we conclude that
\begin{equation}\label{equ:Acond4p}
T\meq{\bigsqcup_{i\in I^p}}^\mu T^i\quad\text{and}\quad 
T=\overline{\bigcup_{i\in I^p}T^i}.
\end{equation}

Now observe that \ref{H1} implies $0\leq\sup_{i\in I^p,\,n\in N}\alpha^i_n\le q^p$.
In particular, $\gamma^i$ is a $q^p$\nobreakdash-Lipschitz function. Then $\mathrm{diam}\,T^i\le q^p\mathrm{diam}\,T=q^p\delta$, and hence
\begin{equation}\label{equ:Acondition10}
\sup_{i\in I^p}\mathrm{diam}\,T^i \le q^p\delta.
\end{equation}

We begin with proving \cref{mainI}.

Note that $a^i=\gamma^i(0)\in\gamma^i(T)=T^i$ and $a^i\in \Gamma^p_\pi(0)\subseteq \mathbf{S}$ for any $p\in\mathbb{N}_0$ and $i\in I^p$ (recall that $a^{\varnothing}_n=0$ for $p=0$). Therefore, \eqref{equ:Acond4p} and \eqref{equ:Acondition10} imply
\begin{equation*}\label{equ:Acondition11}
\textstyle T=\overline{\left\{a^i:i\in \bigcup_{p=0}^{\infty}I^p\right\}}
\subseteq\overline{\mathbf{S}}=\mathbf{T}.
\end{equation*}
	
Now, we pass to the proof of~\ref{mainII}.
	
Let $g\colon X\to\mathbb{R}$ be $N$\nobreakdash-dimensional Lipschitz on every compact subset of $X$. Fix $x\in X$. Pick some $a=(a_\ell)_{\ell\in L}\in V^N$ and $b=(b_\ell)_{\ell\in L}\in V^N$ such that $T\cup(T+x)\subseteq P(a,b)$ and $a_\ell< b_\ell$ for every $\ell\in L$. Set $X_0=X\cap P(a,b)$. Since $X_0$ is compact, we conclude that $g$ is $N$\nobreakdash-dimensional Lipschitz on~$X_0$ with a constant $C\ge 0$. 

Observe that for all $i\in I^p$ and	$n\in N$, we have
\begin{equation*}
\gamma^i(0)=a^i\in T\quad\text{and}\quad \gamma^i_n(x_n)=a^i_n+\alpha^i_nx_n\in[a^i_n,a^i_n+x_n],
\end{equation*}
because $0<\alpha^i_n<1$. In particular, $\gamma^i(0),\gamma^i(x)\in T\cup(T+x)\subseteq P(a,b)$. Since $X$ is admissible, the set $\Pi\big(\gamma^i(0),\gamma^i(x)\big)$ is included in~$X$, and then also contained in $X_0$. Therefore, using \eqref{equ:Acondition8}, we obtain
\begin{align*}
\sum_{i\in I^p}\left|\square g\big(\gamma^i(0),\gamma^i(x)\big)\right|
&\le  \sum_{i\in I^p} C\prod_{n\in N}\big\|\gamma^i_n(x_n)-\gamma^i_n(0)\big\|\\
=C\sum_{i\in I^p}\prod_{n\in N}\alpha^i_n\,\|x_n\|
&=C\sum_{i\in I^p}\beta^i\|x\|^{N} =C\|x\|^{N}<\infty.
\end{align*}
Taking now into account \cref{prop:IterationsMWoperatorsSpecial}, we see that $g\in\mathcal{F}^\infty(X)$. Moreover, by \Cref{prop:NdimLipschitzOfCNfunction}, we conclude that $\widetilde{f}$ is $N$\nobreakdash-dimensional Lipschitz on the compact set $P(a,b)$. In consequence, $\widetilde{f}\in \mathcal{F}^\infty(V^N)$, and hence $f\in \mathcal{F}^\infty(X)$. 
	
Next, we prove~\ref{mainIII}. 

Fix $x\in X$ and $p\in \mathbb{N}$. Recall that $\delta=\text{diam}\,T$ and let $r_x=\max\{\delta,\|x\|\}$. It is enough to show that 
\begin{equation}\label{equ:estimationMp_minus_L}
|\mathbb{M}^pf(x)-\mathbb{L}\widetilde{f}(x)|\le2\|x\|^N\omega_T\big(\nabla_N\widetilde{f},r_xq^p\big).
\end{equation}
	
Using the definition of the average $N$\nobreakdash-gradient and \ref{H2} we conclude that 
\begin{align*}
\mathbb{L}\widetilde{f}(x)&=\overline{\nabla}\!_{N}\widetilde{f}(T) x^N
=\tfrac1{\mu(T)}\widetilde{\nabla}\!_{N}\widetilde{f}(T) x^N
=\tfrac1{\mu(T)}\int_T\nabla\!_N\widetilde{f}(t)x^N\,dt\\
&=\tfrac1{\mu(T)}\sum_{i\in I^p}\int_{T^i}\nabla\!_N\widetilde{f}(t)x^N\,dt
=\sum_{i\in I^p}\tfrac{\mu(T^i)}{\mu(T)}\overline{\nabla}\!_N\widetilde{f}(T^i)x^N\\
&=\sum_{i\in I^p}\beta^i\overline{\nabla}\!_N\widetilde{f}(T^i)x^N.
\end{align*}

Fix $i\in I^p$ and denote by $B^i=B[a^i,\delta q^p]$ the closed ball with centre~$a^i$ and radius $\delta q^p$. Since $a^i\in T^i$ and ${\text{diam}\,T^i\le\delta q^p}$ established in~\eqref{equ:Acondition10}, we see that $B^i$ is a connected superset of $T^i$. Put $J^i=\overline{\nabla}\!_N\widetilde{f}(T^i)x^N$.
Since 
\begin{equation*}
\min_{t\in B^i}\nabla\!_N\widetilde{f}(t) x^N\le \min_{t\in T^i}\nabla\!_N\widetilde{f}(t) x^N \le J^i\le \max_{t\in T^i}\nabla\!_N\widetilde{f}(t) x^N
\le \max_{t\in B^i}\nabla\!_N\widetilde{f}(t) x^N
\end{equation*} 
and the function $t\mapsto\nabla\!_N\widetilde{f}(t)x^N$ is continuous (as $\widetilde{f}$ is a $\mathrm{C}^N$\nobreakdash-function), we conclude that there exists $t^i\in B^i$ such that $J^i=\nabla\!_N\widetilde{f}(t^i)x^N$. Therefore,
\begin{equation}\label{equ:Acondition12}
\mathbb{L}\widetilde{f}(x)=\sum_{i\in I^p}\beta^i \nabla\!_N\widetilde{f}(t^i)x^N.
\end{equation} 
	
By \Cref{prop:IterationsMWoperatorsSpecial,prop:finIncrementNdim}, for some $s^i\in P\big(\gamma^i(0),\gamma^i(x)\big)=\prod_{n\in N}[a^i_n,a^i_n+\alpha^i_nx_n]$ we have
\begin{equation}\label{equ:Acondition13}
\begin{split}
\mathbb{M}^pf(x)&=\sum_{i\in I^p}\square \widetilde{f}\big(\gamma^i(0),\gamma^i(x)\big)
=\sum_{i\in I^p} \nabla\!_N\widetilde{f}(s^i) \big(\gamma^i(x)-\gamma^i(0)\big)^N\\
&= \sum_{i\in I^p}\sum_{k\in K^N}\widetilde{\partial}_k\widetilde{f}(s^i)\odot^N_k(\alpha^i_nx)
=\sum_{i\in I^p}\sum_{k\in K^N}\widetilde{\partial}_k\widetilde{f}(s^i)\prod_{n\in N}\alpha^i_nx_{n,k_n}\\
&=\sum_{i\in I^p}\prod_{n\in N}\alpha^i_n\sum_{k\in K^N}\widetilde{\partial}_k\widetilde{f}(s^i)\prod_{n\in N}x_{n,k_n}
= \sum_{i\in I^p} \beta^i\nabla\!_N\widetilde{f}(s^i) x^N.
\end{split}
\end{equation}
Recalling that $t^i\in B^i$, we have
\begin{equation*}
\|t^i-a^i\|\le \delta q^p\le r_xq^p.
\end{equation*}
Since $s^i,a^i\in P\big(\gamma^i(0),\gamma^i(x)\big)$ and $\gamma^i$ is $q^p$\nobreakdash-Lipschitz, we conclude that  
\begin{equation*}
\|s^i-a^i\|\le\mathrm{diam}\,P\big(\gamma^i(0),\gamma^i(x)\big)
\leq q^p\mathrm{diam} (P(0,x))\le \|x\| q^p\le r_xq^p.
\end{equation*}	

Finally, taking into account \eqref{equ:Acondition13}, \eqref{equ:Acondition12},  \eqref{equ:CBSinequality}, and \eqref{equ:Acondition8}, we obtain
\begin{align*}
|\mathbb{M}^pf(x)-\mathbb{L}\widetilde{f}(x)|&=
\left|\sum_{i\in I^p}\left(\beta^i \nabla\!_N\widetilde{f}(s^i) x^N
-\beta^i \nabla\!_N\widetilde{f}(t^i) x^N\right)\right|\\
&=\left|\sum_{i\in I^p}\beta^i\left(\nabla\!_N\widetilde{f}(s^i)
-\nabla\!_N\widetilde{f}(t^i)\right) x^N\right|\\
&\le \sum_{i\in I^p}\beta^i\left\|\nabla\!_N\widetilde{f}(s^i)
-\nabla\!_N\widetilde{f}(t^i)\right\|\|x\|^N\\
&\le \sum_{i\in I^p}\beta^i\left(\left\|\nabla\!_N\widetilde{f}(s^i)
-\nabla\!_N\widetilde{f}(a^i)\right\|+\left\|\nabla\!_N\widetilde{f}(a^i)- \nabla\!_N\widetilde{f}(t^i)\right\|\right)\|x\|^N\\
&\le \sum_{i\in I^p}\beta^i\,2\omega_T(\nabla_N\widetilde{f},r_xq^p)\|x\|^N
=2\|x\|^N\omega_T(\nabla_N\widetilde{f},r_xq^p),
\end{align*}
which proves \eqref{equ:estimationMp_minus_L}, and so \ref{mainIII} holds.

It remains to prove \ref{mainIV}. 

Let $f\colon X\to \mathbb{R}$ be a $\mathrm{C}^N$\nobreakdash-function such that  $f=\mathbb{M}f$ and $\widetilde{f}\colon V^N\to \mathbb{R}$ be a $\mathrm{C}^N$\nobreakdash-extension of $f$. Then $f(x)=\mathbb{M}^pf(x)$ for all $x\in X$ and $p\in\mathbb{N}$.
Put $\lambda=\overline{\nabla}\!_N\widetilde f(T)$. By~\ref{mainIII}, we conclude that
\begin{equation*}
f(x)=\lim_{p\to\infty}\mathbb{M}^p f(x)=\mathbb{L}\widetilde{f}(x)=\lambda x^N
\end{equation*}
for any $x\in X$. Thus $f$ is of the desired form.
	
Suppose now that there exists $\lambda\in\mathbb R^{K^N}$ such that $f(x)=\lambda x^N$ for any $x\in V^N$. Fix some $x\in V^N$. Then
\begin{equation*}
f(x)=\sum_{k\in K^N}\lambda_k\prod_{n\in N}x_{n,k_n}.
\end{equation*}  
Fix $i\in I$ and $M\subseteq N$. Put $x^i_M=\pi_{M}\big(\gamma^i(0),\gamma^i(x)\big)$ and $b^i_n=\alpha^i_nx_n+a^i_n$ for every $n\in N$. Therefore, $x^i_M=(x^i_{M,n})_{n\in N}$, where $x^i_{M,n}=a^i_n$ if $n\in M$ and $x^i_{M,n}=b^i_n$ if $n\in N\setminus M$. Then
\begin{equation*}
f\Big(\pi_M\big(\gamma^i(0),\gamma^i(x)\big)\Big)
=\sum_{k\in K^N}\lambda_k\prod_{n\in M}a^i_{n,k_n}\prod_{n\in N\setminus M}b^i_{n,k_n}, 
\end{equation*}
and hence
\begin{align*}
\square f\big(\gamma^i(0),\gamma^i(x)\big)&=\sum_{M\subseteq N}(-1)^{|M|}
f\left(\pi_M\big(\gamma^i(0),\gamma^i(x)\big)\right)\\
&=\sum_{M\subseteq N}(-1)^{|M|}
\sum_{k\in K^N}\lambda_k\prod_{n\in M}a^i_{n,k_n}\prod_{n\in N\setminus M}b^i_{n,k_n} \\
&=\sum_{k\in K^N}\lambda_k\sum_{M\subseteq N}(-1)^{|M|}
\prod_{n\in M}a^i_{n,k_n}\prod_{n\in N\setminus M}b^i_{n,k_n} \\
&=\sum_{k\in K^N}\lambda_k\prod_{n\in N}(b^i_{n,k_n}-a^i_{n,k_n})
=\sum_{k\in K^N}\lambda_k\prod_{n\in N}\alpha^i_{n}x^i_{n,k_n}\\
&=\beta^i\sum_{k\in K^N}\lambda_k\prod_{n\in N}x^i_{n,k_n}
=\beta^i\lambda x^N.
\end{align*}
Finally, since $\sum\limits_{i\in I}\beta^i=1$, we obtain
\begin{equation*}
\mathbb{M}f(x)=\sum_{i\in I}\square f\big(\gamma^i(0),\gamma^i(x)\big)
=\sum_{i\in I}\beta^i\lambda x^N=\lambda x^N=f(x),
\end{equation*}
and the proof is complete.
\end{proof}

In the case where $s=1$ assertion \ref{mainIV} of \cref{thm:MWequatoC1} takes the form
\begin{enumerate}
\item[(\textit{iv'})] $\mathbb{M}f=f$ if and only if there exists $\lambda\in\mathbb{R}$ such that  $f(x)=\lambda x^N=\lambda\prod\limits_{n\in N}x_n$ for every $x\in X$.
\end{enumerate}


\section{Multivariate distribution of a Borel measure}\label{section10}
To the end of this paper, we assume that $s=1$, i.e. $V=\mathbb R$.

For all $x,y\in V$ we put $[x,y)=[x,y]\setminus\{y\}$, and for all $a=(a_n)_{n\in N},b=(b_n)_{n\in N}\in [0,\infty)^N$ we define $Q(a,b)=\prod_{n\in N}[a_n,b_n)$.
Given $a,b\in [0,\infty)^N$ we write $a\preceq b$ if $a\in P(0,b)$, and 
$a\prec b$ if $a\in Q(0,b)$. Obviously, for all $a,b\in [0,\infty)^N$ with $a\preceq b$, we have
\begin{equation*}
P(a,b)=\big\{x\in [0,\infty)^N:a\preceq x\preceq b\big\}\quad\text{and}\quad
Q(a,b)=\big\{x\in [0,\infty)^N:a\preceq x\prec b\big\}.
\end{equation*} 

To the end of this paper, we assume the following hypothesis.
\begin{enumerate}[label={\rm (H$_3$)}]
  \item\label{H3} Assume that $T\subseteq [0,\infty)^N$ is compact and 
$Q(0,x)\subseteq T$ for every $x\in T$.
\end{enumerate}
Set 
\begin{equation*}
\mathcal{S}(T)=\big\{Q(a,b):a,b\in T\text{ with }a\preceq b\big\}.
\end{equation*}
It is easy to check that $\mathcal{S}(T)$ is a semi-ring of sets and $\mathcal{S}(T)$  generates $\mathcal{B}(T)$.

Let $\nu$ be a Borel measure on $T$. The \textit{multivariate distribution} of  $\nu$ is said to be the function $d_\nu\colon T\to\mathbb{R}$ defined by 
\begin{equation*}
d_\nu(x)=\nu(Q(0,x)).
\end{equation*}
We say that a Borel measure $\nu$ on $T$ is a \textit{$\mathrm{C}^N$-measure} if its multivariate distribution $d_\nu$ is a $\mathrm{C}^N$-function. Clearly, the multivariate distribution of the Lebesgue measure $\mu$ on $T$ is of the form
\begin{equation}\label{equ:distribution_of_Lebesgue_measure}
d_\mu(x)=x^N=\prod_{n\in N}x_n\quad\text{for every }x\in T;
\end{equation}
note that $T\subseteq[0,\infty)^N$ is crucial here. Consequently, $d_\mu$ is a $\mathrm{C}^N$-measure on $T$.

\begin{proposition}\label{prop:measureOfNinterval}
Assume~\cref{H3}, that $\nu$ is a Borel measure on $T$, and let $a,b\in T$ with $a\preceq b$. Then
$\nu(Q(a,b))=\square d_\nu(a,b)$.
\end{proposition}

\begin{proof}
In the case $r=1$, the condition $a\preceq b$ means that $a=tb$ for some $t\in[0,1]$. Then $Q(a,b)=[a,b)=[0,b)\setminus [0,a)=Q(0,b)\setminus Q(0,a)$, and hence 
\begin{equation*}
\nu(Q(a,b))=\nu(Q(0,b))-\nu(Q(0,a))=d_\nu(b)-d_\nu(a)=\square d_\nu(a,b).
\end{equation*}

Suppose now that for every $p\in\{1,\ldots,r-1\}$ the claim of the proposition holds for the set $N_p=\{1,\ldots,p\}\subseteq N$. Fix $a,b\in T$ such that $a\preceq b$. Put $N'=N_{r-1}$, $a'=a|_{N'}$, $b'=b|_{N'}$, and $T'=\{x'\in V^{N'}:(b_r,x')_r\in T\}$. Then $0\le a_r\le b_r$ and $a'\preceq b'$. Moreover, if $x'\in T'$ and $y'\in P(0,x')$, then $(b_r,x')_r\in T$ and \cref{H3} implies that $(b_r,y')_r\in Q(0,(b_r,x')_r)\subseteq T$. Hence, $y'\in T'$, and thus $T'$ satisfies the same properties as~$T$ as far as \cref{H3} is concerned (and $N'$ instead of~$N$).

Since $b'\in T'$, we have $Q(0,a')\subseteq Q(0,b')\subseteq T'$, and hence
\begin{align*}
Q(a,b)&=[a_r,b_r)\times_r Q(a',b')=\big([0,b_r)\setminus[0,a_r)\big)\times_r Q(a',b')\\	 
&=\big([0,b_r)\times_r Q(a',b')\big)\setminus\big([0,a_r)\times_r Q(a',b')\big).
\end{align*}

Defining two Borel measures $\nu_1$ and $\nu_2$ on~$T'$ by
$\nu_1(E)=\nu([0,a_r)\times_r E)$ and $\nu_2(E)=\nu([0,b_r)\times_r E\big)$, we obtain
\begin{equation}\label{equ:measureOfNinterval1}
\nu(Q(a,b))=\nu_2(Q(a',b'))-\nu_1(Q(a',b')).
\end{equation}
Then 
\begin{equation*}
d_{\nu_1}(x')=\nu_1(Q(0,x'))=\nu([0,a_r)\times_r Q(0,x'))
=\nu\big(Q(0,(a_r,x')_r)\big)=d_\nu((a_r,x')_r).
\end{equation*}
for any $x'\in T'$, and hence $d_{\nu_1}=(d_{\nu})_{a_r{:}r}$. In the same way we obtain $d_{\nu_2}=(d_{\nu})_{b_r{:}r}$. 

Finally, using \eqref{equ:measureOfNinterval1}, the inductive hypothesis and \Cref{prop:indIncrement}, we conclude that 
\begin{align*}
\nu(Q(a,b))&=\nu_2(Q(a',b'))-\nu_1(Q(a',b'))
=\square d_{\nu_2} (a',b')-\square d_{\nu_1} (a',b')\\
&=\square (d_{\nu})_{b_r{:}r}(a',b')-\square (d_{\nu})_{a_r{:}r}(a',b')
=\square d_{\nu}(a,b),
\end{align*}
and the proof is complete.
\end{proof}

\begin{proposition}\label{prop:equality_of_measure_with_the_same_distr}
Assume~\cref{H3} and that $\nu_1,\nu_2$ are Borel measures on $T$ such that $d_{\nu_1}(x)=d_{\nu_2}(x)$ for every $x\in T$. Then $\nu_1=\nu_2$.
\end{proposition}

\begin{proof}
Fix $a,b\in T$ with $a\preceq b$.
By \Cref{prop:measureOfNinterval}, we have
\begin{equation*}
\nu_1(Q(a,b))=\square d_{\nu_1}(a,b)=\square d_{\nu_2}(a,b)=\nu_2(Q(a,b)),
\end{equation*}
which shows that $\nu_1$ and $\nu_2$ coincide on $\mathcal{S}(T)$. Since the family $\{E\in\mathcal{B}(T):\nu_1(E)=\nu_2(E)\}$ is a $\sigma$-field containing $\mathcal{S}(T)$, we see that it equals $\mathcal{B}(T)$. Therefore, $\nu_1=\nu_2$.
\end{proof}

\begin{proposition}\label{prop:absolutely_continuity_of_CN_measures}
Assume~\cref{H3} and let $\nu$ be a $\mathrm{C}^N$\nobreakdash-measure on $T$ and such that $C=\sup_{x\in T} |\partial_N d_\nu(x)|$. Then $\nu(P)\le C\mu(P)$ for every $P\in\mathcal{S}(T)$.
In particular, $\nu$ is $\mu$-absolutely continuous.
\end{proposition}

\begin{proof}
Fix distinct points $a,b\in T$ with $a\preceq b$ and put $P=Q(a,b)$. By \cref{H3}, we have $P\subseteq T$, and hence $\overline{P}=P(a,b)\subseteq \overline{T}=T$. Applying \Cref{prop:measureOfNinterval,prop:finIncrementNdim}, we deduce that there exists $c\in \overline{P}\subseteq T$ such that
\begin{equation*}
\nu(P)=\square d_\nu(a,b)=\nabla\!_Nd_\nu(c)(b-a)^N=\partial_Nd_\nu(c)(b-a)^N,
\end{equation*}
and since $\mu(P)=(b-a)^N$, we obtain
\begin{equation*}
\nu(P)=\partial_Nd_\nu(c)(b-a)^N= \partial_Nd_\nu(c)\mu(P)\le C\mu(P),
\end{equation*}
which completes the proof.
\end{proof}


\section{Invariant measures under some class of piecewise affine functions}\label{section11}
Recall that we assumed  $s=1$, i.e.\ $V=\mathbb R$, and $D=V^N$. We also assume the hypotheses~\cref{H1,H2,H3}.

Recall that a Borel measure $\nu$ on $T$ is said to be an \textit{invariant measure} under a Borel function $g\colon T\to T$ (or briefly \textit{$g$-invariant}) if
\begin{equation*}
\nu(E)=\nu(g^{-1}(E))\quad\text{for every }E\in\mathcal{B}(T).
\end{equation*}

Put
\begin{equation*}\label{equ:definition_of_the_family_Ai}
A=\bigcup_{i,j\in I}(\gamma^i)^{-1}(\gamma^i(T)\cap \gamma^j(T))\quad\text{and}\quad
A^i=\gamma^i(A)\quad\text{for every }i\in I.
\end{equation*}
By \ref{H2}, we have
\begin{equation}\label{equ:r_disjointness_of_gamma1}
T\mathop{=}^\mu A\mathop{=}^\mu \bigsqcup_{i\in I} A^i,
\end{equation}
and we define the Borel function $g_\gamma\colon T\to T$ by
\begin{equation}\label{equ:definition_of_g_gamma}
g_\gamma(x)=\begin{cases}
(\gamma^i)^{-1}(x),&\text{if }x\in A^i\text{ for some }i\in I,\\
0,&\text{if }x\in T\setminus\bigsqcup_{i\in I}A^i.
\end{cases}
\end{equation}

\begin{theorem}\label{thm:IM} 
Assume \ref{H1}, \ref{H2}, and \ref{H3}. Let $\nu$ be a $\mathrm{C}^N$-measure on $T$. Then the following conditions are equivalent:
\begin{enumerate}[label=$(\roman*)$]
\item\label{equ:IM(i)} $\nu$ is $g_\gamma$-invariant; 
\item\label{equ:IM(ii)}  $\mathbb{M}d_\nu=d_\nu$;
\item\label{equ:IM(iii)} there exists $\lambda\ge 0$ such that $\nu=\lambda\mu$.
\end{enumerate}
\end{theorem}

\begin{proof}  
Let us begin with some observations and notations. Observe first that by \ref{H1} for any $i\in I$ the function  $\gamma^i\colon V^N\to V^N$ is a bijection and
\begin{equation}\label{equ:image_of_inerval_by_gamma_i}
\gamma^i(Q(a,b))=Q(\gamma^i(a),\gamma^i(b))\quad\text{for every }a,b\in V^N.
\end{equation} 

Define the Borel measure~$\nu_\gamma$ on~$T$ by 
\begin{equation*}
\nu_\gamma(E)=\nu(g_\gamma^{-1}(E)).
\end{equation*}
Let $f=d_\nu$ and $f_\gamma=d_{\nu_\gamma}$. Fix $E\in \mathcal B(T)$. Put $E^i=\gamma^i(E)$ for every $i\in I$. Set $B=T\setminus\bigsqcup_{i\in I}A^i$. According to \eqref{equ:r_disjointness_of_gamma1}, we have $\mu(B)=0$ and $\mu(\gamma^i(T)\setminus A^i)=0$ for every $i\in I$. Put $B_0=B$ if $0\in E$ and $B_0=\emptyset$ if $0\notin E$. Then $g_\gamma^{-1}(E)\cap B=B_0$, $g_\gamma^{-1}(E)\cap A^i=E^i\cap A^i$ for all $i\in I$, and
\begin{equation*}
g_\gamma^{-1}(E)=(g_\gamma^{-1}(E)\cap B)
\sqcup\bigsqcup_{i\in I}(g_\gamma^{-1}(E)\cap A^i)
=B_0\sqcup\bigsqcup_{i\in I}(E^i\cap A^i).
\end{equation*}

By \Cref{prop:absolutely_continuity_of_CN_measures}, $\nu$ is $\mu$-absolutely continuous. Thus, $\nu(B_0)=0$ and $\nu(E^i\cap A^i)=\nu(E^i\cap \gamma^i(T))=\nu(E^i)$ for any $i\in I$, and hence
\begin{equation*}
\nu_\gamma(E)=\nu(g_\gamma^{-1}(E))=\nu(B_0)+\sum_{i\in I}\nu(\gamma^i(E)\cap A^i)
=\sum_{i\in I}\nu(E^i).
\end{equation*}
Therefore,
\begin{equation}\label{equ:preimage_of_Borel_by_g_gamma}
\nu_\gamma(E)=\sum_{i\in I}\nu(\gamma^i(E))\quad\text{for every }E\in\mathcal{B}(T).
\end{equation}

Fix $x\in T$. Put $Q=Q(0,x)$ and $Q^i=Q(\gamma^i(0),\gamma^i(x))$ for every $i\in I$. 
Then \eqref{equ:image_of_inerval_by_gamma_i} implies that $Q^i=\gamma^i(Q)$.
Further, we conclude by~\cref{H3} that $Q\subseteq T$.
Using now \eqref{equ:preimage_of_Borel_by_g_gamma} with $E=Q$ we obtain
\begin{equation}\label{f}
f_\gamma(x)=\nu_\gamma(Q)=\sum_{i\in I}\nu(Q^i).
\end{equation}
\Cref{prop:measureOfNinterval} implies that $\nu(Q_i)=\square f(\gamma^i(0),\gamma^i(x))$, which gives
\begin{equation}\label{Mf}
\mathbb{M}f(x)=\sum_{i\in I}\square f(\gamma^i(0),\gamma^i(x))=\sum_{i\in I}\nu(Q^i).
\end{equation}
Finally, combining \eqref{f} with \eqref{Mf} we obtain
\begin{equation}\label{equ:distribution_of_phi_gamma}
f_\gamma=\mathbb{M}f.
\end{equation}

Let us pass to the proof of the equivalences of \ref{equ:IM(i)}, \ref{equ:IM(ii)}, and \ref{equ:IM(iii)}.

\ref{equ:IM(i)}$\Rightarrow$\ref{equ:IM(ii)}.  Let $\nu$ be $g_\gamma$-invariant. Then  $\nu=\nu_\gamma$ and, so $f=f_\gamma$. In consequence,  \eqref{equ:distribution_of_phi_gamma} gives	$f=\mathbb{M} f_\gamma$.

\ref{equ:IM(ii)}$\Rightarrow$\ref{equ:IM(i)}.  Let us assume that $\mathbb{M}f=f$. Then, by \eqref{equ:distribution_of_phi_gamma}, we have $f=\mathbb{M} f=f_\gamma$. This jointly with \Cref{prop:equality_of_measure_with_the_same_distr} yields $\nu=\nu_\gamma$, and hence
$\nu$ is $g_\gamma$-invariant.

\ref{equ:IM(ii)}$\Rightarrow$\ref{equ:IM(iii)}. Using \cref{thm:MWequatoC1} with $s=1$ and $X=T$ (in this case assertion (\textit{iv}) reduces to (\textit{iv'}) stated at the end of \cref{section9}), we conclude that there exists $\lambda\in \mathbb{R}$ such that $f(x)=\lambda x^N$ for any $x\in T$. Since $f\ge 0$, we get $\lambda\ge 0$. This together with \eqref{equ:distribution_of_Lebesgue_measure} implies that $f=\lambda d_\mu$. Finally, by \Cref{prop:equality_of_measure_with_the_same_distr}, we obtain  $\nu=\lambda\mu$.

\ref{equ:IM(iii)}$\Rightarrow$\ref{equ:IM(ii)}. Assume that there exist $\lambda\ge 0$ such that $\nu=\lambda\mu$. Then \eqref{equ:distribution_of_Lebesgue_measure} implies that $f(x)=d_\nu(x)=d_{\lambda\mu}(x)=\lambda x^N$ for every $x\in T$. Finally, applying  \cref{thm:MWequatoC1} \ref{mainIV} with $s=1$ and $X=T$, we conclude that $f=\mathbb{M}f$.
\end{proof}

We end this paper with the following question.

\begin{problem}
Do the equivalences \ref{equ:IM(i)}$\Leftrightarrow$\ref{equ:IM(ii)}$\Leftrightarrow$\ref{equ:IM(iii)} in \Cref{thm:IM} hold for any Borel measure $\nu$ that is absolutely continuous with respect to the Lebesgue measure $\mu$ on $T$?
\end{problem}


\section*{Acknowledgment}
The research was supported by the University of Silesia, Institute of Mathematics (Iterative Functional Equations and Real Analysis program).

\section*{Disclosure statement}
The authors declare no competing interests.

\section*{Notes on contributors} 
All work regarding this manuscript was carried out by a joint effort of the three authors.

\input{MMZ.bbl}

\end{document}

%% file: MMZ.bbl
\providecommand{\bysame}{\leavevmode\hbox to3em{\hrulefill}\thinspace}
\providecommand{\MR}{\relax\ifhmode\unskip\space\fi MR }
\providecommand{\MRhref}[2]{%
  \href{http://www.ams.org/mathscinet-getitem?mr=#1}{#2}
}
\providecommand{\href}[2]{#2}